\documentclass[hidelinks,onefignum,onetabnum]{siamart220329}

\usepackage{lipsum}
\usepackage{amsfonts}
\usepackage{graphicx}
\usepackage{epstopdf}
\ifpdf
  \DeclareGraphicsExtensions{.eps,.pdf,.png,.jpg}
\else
  \DeclareGraphicsExtensions{.eps}
  \fi

\graphicspath{{figs/}}

\usepackage{commath}
\usepackage{mathrsfs}
\usepackage{mathtools}
\usepackage{amssymb}
\usepackage{caption}

\usepackage{algorithm}
\usepackage{algpseudocode}

\newcommand{\R}{\mathbb{R}}

\newcommand{\normtwo}[1]{\norm{#1}_2}

\makeatletter
\newcommand{\fnorm}{\@ifstar\@fnorms\@fnorm}
\newcommand{\@fnorms}[1]{%
  \left|\mkern-1.5mu\left|\mkern-1.5mu\left|
   #1
  \right|\mkern-1.5mu\right|\mkern-1.5mu\right|
}
\newcommand{\@fnorm}[2][]{%
  \mathopen{#1|\mkern-1.5mu#1|\mkern-1.5mu#1|}
  #2
  \mathclose{#1|\mkern-1.5mu#1|\mkern-1.5mu#1|}
}

\newcommand{\fnormtwo}{\@ifstar\@fnormstwo\@fnormtwo}
\newcommand{\@fnormtwo}[2][]{\@fnorm[#1]{#2}_2}
\newcommand{\@fnormstwo}[1]{\@fnorms{#1}_2}
\makeatother

\newcommand{\hide}[1]{}

\newsiamremark{remark}{Remark}
\newsiamremark{hypothesis}{Hypothesis}
\crefname{hypothesis}{Hypothesis}{Hypotheses}
\newsiamthm{claim}{Claim}

\newsiamremark{assumption}{Assumption}

\headers{Cholesky Factorizations of Finite Horizon Gramians}{T. Stillfjord and F. Tronarp}

\title{Computing the matrix exponential and the Cholesky factor of a related finite horizon Gramian
\thanks{Submitted to the editors 2025-05-16\funding{TS and FT were partially supported by the Wallenberg AI, Autonomous Systems and Software Program (WASP) funded by the Knut and Alice Wallenberg Foundation. TS was additionally partially supported by the Swedish Research Council under grants 2023-03982 and 2023-04862.
}}}

\author{Tony Stillfjord\thanks{Centre for Mathematical Sciences, Lund University, Sweden (\email{tony.stillfjord@math.lth.se}, \email{filip.tronarp@matstat.lu.se}).}
  \and
  Filip Tronarp\footnotemark[2]}

\usepackage{amsopn}

\begin{document}

\maketitle

\begin{abstract}
In this article, an efficient numerical method for computing both the matrix exponential and a finite horizon controllability Gramian in Cholesky-factored form is proposed.
The method is applicable to general dense matrices of moderate size and produces a Cholesky factor of the Gramian without computing the full product.
It is a generalization of the scaling-and-squaring approach for approximating the matrix exponential and exploits a similar doubling formula for the Gramian to compute
both at once. The required computational effort is thereby kept modest.
Most importantly, a rigorous backward error analysis is provided, which guarantees that the approximation is accurate to the round-off error level in double precision.
This accuracy is illustrated in practice on a large number of standard test examples.
The method has been implemented in the Julia package FiniteHorizonGramians.jl, which is available online under the MIT license.
Code for reproducing the experimental results is included in this package, as well as code for determining the optimal method parameters.
The analysis can thus easily be adapted to a different finite-precision arithmetic.
\end{abstract}




\section{Introduction}
Consider a pair of matrices $A \in \mathbb{R}^{n \times n}$ and $B \in \mathbb{R}^{n \times m}$.
This article is concerned with the numerical approximation of two matrix functions $\Phi(A)$ and $G(A, B)$, defined by
\begin{subequations}
\begin{align}
\Phi(A) &= e^{A}, \label{eq:exponential} \\
G(A, B) &= \int_0^1 e^{A \tau} BB^* e^{A^* \tau} \dif \tau \label{eq:gramian},
\end{align}
\end{subequations}
where $^*$ denotes the Hermitian conjugate.
The first function, $\Phi$, is just the matrix exponential with specialized notation, which shall simplify the subsequent discussion.
The second function, $G$, is the controllability Gramian of the pair $(A, B)$ over the unit interval, and may equivalently be characterized as
the solution at the end-point of the following Lyapunov differential equation
\begin{equation}\label{eq:LDE}
\dot{Q}(t) = A Q(t) + Q(t) A^* + B B^*, \quad Q(0) = 0, \quad t \in [0, 1],
\end{equation}
that is $G(A, B) = Q(1)$ \cite{AbouKandil2012} and more generally
\begin{equation}\label{eq:gramian-reexpression}
Q(t) = G(A t, B \sqrt{t}) = t \int_0^1 e^{A t \tau} BB^* e^{A^* t \tau} \dif \tau = \int_0^t e^{A \tau} BB^* e^{A^* \tau} \dif \tau.
\end{equation}
The controllability Gramian is always positive semi-definite. If the pair $(A, B)$ is controllable, then it is positive definite \cite{Antsaklis2007}.
Therefore, $G$ has a Cholesky factorization $G(A, B) = U^*(A, B) U(A, B)$ for some upper triangular matrix function $U(A, B)$\footnote{
However, when the Gramian fails to be positive definite, then the Cholesky factor is not unique.
}.

\subsection{Computational issues}
\label{sec:problem-statement}
When $(A, B)$ is not controllable then $G(A, B)$ is singular \cite{Antsaklis2007}.
Furthermore, even if $(A, B)$ controllable it follows from \eqref{eq:gramian-reexpression} that
\begin{equation*}
G(A t, B \sqrt{t}) = B B^* t + o(t).
\end{equation*}
Consequently, the Gramian over the interval $[0, t]$ can become arbitrarily close to singular.
This means that a Cholesky factor can not reliably be obtained in general via the Cholesky algorithm applied to a computed $G$, as the algorithm will inevitably fail when $G$ has small eigenvalues \cite[Theorem 3.2]{Demmel1989}.
Furthermore, the pivoted Cholesky algorithm can not be guaranteed to be numerically stable when $n$ is moderately large and the numerical rank of $G$ is not small \cite[Section 10.3.2]{Higham1996}.
This is precisely the case when the uncontrollable subspace of $(A, B)$ is small and $n$ is moderately large, thus ruling out pivoted Cholesky in those situations.

On the other hand, from \eqref{eq:gramian-reexpression} it can be seen that $G(A t, B \sqrt{t})$ satisfies the following algebraic Lyapunov equation:
\begin{equation}
A G(A t, B \sqrt{t}) + G(A t, B \sqrt{t}) A^* + D(At, B) = 0, \quad   D(A t, B) = B B^* - e^{A t} B B^* e^{A^* t}.  \label{eq:ale}
\end{equation}
If $A$ is Hurwitz (all eigenvalues have negative real part) and $D$ is non-negative definite,
then \eqref{eq:ale} has a unique positive definite solution~\cite[Corollary 1.1.4]{AbouKandil2012}.
The Gramian can the be obtained by first solving \eqref{eq:ale} for $G$ followed by a Cholesky factorization.
However, $D(A t, B)$ can become numerically indefinite for sufficiently small $t$, which prohibits this approach in general.

On the other hand, given a square-root factor $D^{1/2}(A t, B)$ the algorithm by Hammarling~\cite{Hammarling1982} is an efficient means to obtain a Cholesky factor of $G$,
but obtaining $D^{1/2}(A t, B)$ is non-trivial.
Using the existing toolbox of numerical linear algebra,
the strategy would be to obtain a Cholesky factor of $B B^*$ via the QR decomposition and then obtain $D^{1/2}$ by Cholesky downdates,
but the downdate problem can be poorly conditioned and even fail completely \cite{Elden1994,Stewart1979}.
Therefore, due to the Hurwitz requirement on $A$ and the preceding discussion, computing the Gramian via \eqref{eq:ale} is not applicable in general, nor generally reliable when it is applicable.\footnote{
The first experiment of Section \ref{sec:experiments} examines a concrete class of problems where $A$ is not Hurwitz.
}
Thus there does not appear to exist a general purpose method for reliably obtaining a Cholesky factor of $G$.

\subsection{Contribution}
The aim is to develop a numerical algorithm for the computation of both the matrix exponential $\Phi(A)$
and a Cholesky factor $U(A, B)$ without forming $G(A, B)$ as an intermediate step.
This has applications in numerically robust implementations of linear filters and smoothers,
via the so called square-root or array algorithms  \cite{AndersonMoore1979,Kailath2000},
and in robust state-space balancing and truncation algorithms \cite{Antoulas2005}.
The approach proposed here is based on a certain doubling recursion for both $\Phi$ and $G$ \cite{VanLoan1978},
which extends the scaling and squaring method of the matrix exponential \cite{Higham2005, Higham2008, AlMohy2010}.
This is similar to the doubling recursion for computing both the matrix exponential and its Fr\'echet derivative \cite[Section 10.6]{Higham2008},
which was turned into a serviceable algorithm by \cite{AlMohy2009} by computing the Fr\'echet derivative of the initial Pad\'e approximation of the matrix exponential.
However, this provides no information on how to construct the initial approximations which are to be doubled.
That is accomplished here by drawing on a connection between the diagonal Pad\'e approximants and a certain Petrov--Galerkin approximation of $\Phi(A t)$ on the interval $[0, 1]$ \cite{Moore2011}.
This allows for the development of an algorithm that almost computes $\Phi$ in the conventional way while also providing an initial approximation to the Cholesky factor $U$,
both of which are then propagated through the doubling recursion.

The error analysis for $\Phi$ is the same as for the conventional method \cite{Higham2005},
while the error analysis for $G$ is more intricate.
By lifting the problem to the computation of the Gramian functional
\begin{equation}\label{eq:gramian-functional}
\mathcal{G}(\tilde{A}, \tilde{B}) = \int_0^1 e^{ \tilde{A}(t)} \tilde{B}(t) \tilde{B}^*(t) e^{\tilde{A}^*(t)} \dif t,
\end{equation}
such that $G(A, B) = \mathcal{G}(t \mapsto At, t \mapsto B)$, it is demonstrated that the proposed method is backward stable,
in the sense that there are small perturbations $\Delta A(t)$ and $\Delta B(t)$ such that the approximated Gramian is exactly given by
\begin{equation*}
\mathcal{G}(At + \Delta A, B + \Delta B).
\end{equation*}
Norm bounds on $2^{-s} A$ are obtained which guarantee that
the relative errors in the perturbed data $\Delta A$ and $\Delta B$ are bounded by unit round-off in the induced supremum norm.
The norm of $B$ turns out to be irrelevant for the backward errors.
Additionally, the relative condition number of $\mathcal{G}$ is estimated,
thus giving an estimate of the relative forward error \cite[Section 3]{Higham2008}.

As the algorithm closely resembles the classical scaling and squaring algorithm with Pad\'e approximants,
it is expected to be appropriate to use under the same circumstances.
This is not only corroborated by the error analysis but also through numerical experiments.

The proposed algorithm and the analysis informing its design has been implemented
in the Julia programming language \cite{Bezanson2017}.
The resulting software package, FiniteHorizonGramians.jl\footnote{
\url{https://github.com/filtron/FiniteHorizonGramians.jl}.
},
is released under the MIT license.

The rest of the article is organized as follows.
This section concludes with a discussion on related work.
The main ideas behind the algorithm construction are established in Section~\ref{sec:algorithm-sketch}.
In Section~\ref{sec:backward-error-analysis}, a backward error analysis is presented, which is complemented by a forward error analysis based on relative condition number estimates in Section~\ref{sec:forward-error-analysis}. The rank properties of the approximated Gramian in relation to the true Gramian are discussed in Section~\ref{sec:rank}.
The final algorithm design is settled in Section~\ref{sec:algorithm_design}, followed by several numerical experiments in Section~\ref{sec:experiments}.
Concluding remarks are given in Section~\ref{sec:conclusion}.

\subsection{Related work}
The development of methods for computing matrix exponentials have a long history \cite{Moler1978}.
A prominent approach is based on the scaling and squaring method using various base approximations \cite{Moler2003},
where the Pad\'e approximations have become preferred \cite{Higham2005, Higham2008,AlMohy2010,Guttel2016}.
Computing the matrix exponential and some associated quantity has been done by \cite{AlMohy2009}, in the case of the Fr\'echet derivative.
However, to the authors' knowledge no such development has been done for computing Cholesky factors of finite horizon Gramians.
In the following, the literature closest to the present contribution is reviewed.

\paragraph{Full finite horizon Gramians}
A natural but not necessarily good idea is to first compute a full, non-factorized, Gramian and then compute a Cholesky factorization of it.
There are several methods for computing Gramians, as outlined in e.g.~\cite{Axelsson2014}.
The ones most generally applicable utilize Radon's lemma~\cite[Theorem 3.1.1]{AbouKandil2012}
and compute the exponential of a matrix $J$ which is twice as large as $A$, followed by a linear solve.
This approach is sometimes known as the matrix fraction decomposition method.
Van Loan~\cite{VanLoan1978} exploits the block triangular structure of $J$ arriving at the same doubling recursion as in the algorithm proposed in this paper.
However, that work did not consider the factorization of the Gramian, and only provides a forward error analysis.
As mentioned in Section~\ref{sec:problem-statement}, obtaining a Cholesky factor from a computed Gramian can not reliably be done with conventional methods.

\paragraph{Large-scale Lyapunov equations}
Many numerical methods have been suggested for solving large-scale differential Lyapunov equations~\eqref{eq:LDE},
e.g.\ BDF and Rosenbrock methods~\cite{BennerMena2016,Mena07}, projection methods~\cite{KoskelaMena2020,BehrBennerHeiland2019,KirstenSimoncini2020},
exponential integrators~\cite{LiZhangLiu2021}, splitting schemes~\cite{OstermannPiazzolaWalach2019,Stillfjord2015,Stillfjord2018},
and numerical quadrature~\cite{Stillfjord2015,Stillfjord2018}. Most of these methods are designed for differential Riccati equations,
but they reduce to methods for Lyapunov equations by setting the nonlinear term to zero.
The focus in all these works has been on low-rank approximations $G(A,B) \approx U^* U$,
where $U^* \in \R^{n \times r}$ with $r \ll n$ and large $n$.
With the exception of the Krylov methods, these methods lack rigorous error analyses.
In practice, they typically produce approximations with errors in the range $[10^{-3}, 10^{-12}]$.
This is in contrast to the present interest,
which is to obtain full-rank upper triangular Cholesky factors $U \in \R^{n \times n}$, for $n$ moderately small,
which are fully accurate in finite precision arithmetic.

The low-rank approach has also been used for large-scale algebraic Lyapunov equations, and the representation~\eqref{eq:ale} was used by e.g.\ K\"{u}rschner~\cite{Kuerschner2018} to compute finite horizon Gramians. For an overview of the wide spread of methods, see the surveys~\cite{BennerSaak2013,Simoncini2016}.
Most notable is, perhaps, the commonly used LRCF-ADI method~\cite{BennerLiPenzl2008,LiWhite2002}.

\section{Sketch of algorithm}\label{sec:algorithm-sketch}
In this section, the main ideas of the proposed algorithm are established.
In particular, the doubling formula for both $\Phi$ and $G$ is established in Section~\ref{subsec:doubling}.
The doubling recursion requires initial approximations of $\Phi$ and $G$, and such approximations based on a
Petrov--Galerkin method in a shifted Legendre basis are reviewed in Section~\ref{subsec:initialization}.

\subsection{Doubling formulae}\label{subsec:doubling}
It is well known that the matrix exponential satisfies the doubling formula $\Phi(A) = \Phi^2(A/2)$,
which is a crucial component of the standard algorithm for numerically computing matrix exponentials \cite{Higham2005}.
It is perhaps less known that $G$ also satisfies a doubling formula, but such ideas have been around for a long time \cite[Section 6.7]{AndersonMoore1979}, \cite{VanLoan1978}.
The result is as follows.
\begin{lemma}
The function $G(A,B)$ satisfies the following doubling formula
\begin{equation*}
G(A,B) = G(A/2,B/\sqrt{2}) + e^{A/2} G(A/2,B/ \sqrt{2}) e^{A^* /2}.
\end{equation*}
\end{lemma}
\begin{proof}
Split the integral defining $G(A,B)$ in half:
\begin{equation*}
\begin{split}
G(A,B) &=
\int_0^{1/2} e^{A \tau} B B^* e^{A^* \tau} \dif \tau +
e^{A/2}\int_{1/2}^1 e^{A (\tau - 1/2)} B B^* e^{A^* (\tau - 1/2)} \dif \tau \,e^{A^*/2} \\
&= G(A/2,B/\sqrt{2}) + e^{A/2} G(A/2,B/\sqrt{2}) e^{A^*/2}.
\end{split}
\end{equation*}
Here, the last step is a change of variables $\tau \mapsto 2( \tau - 1/2 )$ in the second integral.
\end{proof}
Now let $s$ be an integer and define the following series of functions for $k=0,1,\ldots, s$:
\begin{align*}
\Phi_k(A) &= \Phi(A / 2^{s-k}), \\
G_k(A, B) &= G( A / 2^{s-k}, B / \sqrt{2^{s-k}}).
\end{align*}
From the doubling formulae, a recursion for $\Phi_k$ and $G_k$ is readily obtained:
\begin{subequations} \label{eq:exp-and-gram-doubling}
\begin{align}
\Phi_{k+1}(A) &= \Phi_k^2(A), \label{eq:exp-doubling} \\
G_{k+1}(A, B) &= \Phi_k(A) G_k(A, B)  \Phi_k^*(A) + G_k(A, B). \label{eq:gram-doubling}
\end{align}
\end{subequations}
This suggests that an efficient algorithm for approximating $\Phi$ and $G$ may be obtained by simply
adding some extra computations to the scaling and squaring algorithm for the matrix exponential.
However, the goal is to compute a Cholesky factor of the Gramian, which can be achieved by factoring $G_k(A,B) = U_k(A,B)^*U_k(A,B)$. By~\eqref{eq:gram-doubling}, the factor $U_k(A,B)$ satisfies the recursion
\begin{equation*}
U_{k+1}^*(A,B)U_{k+1}(A,B) =
\begin{bmatrix}  U_k(A,B) \Phi_k^*(A) \\ U_k(A,B) \end{bmatrix}^*
\begin{bmatrix}  U_k(A,B) \Phi_k^*(A) \\ U_k(A,B) \end{bmatrix},
\end{equation*}
and $U_{k+1}(A,B)$ may then be obtained by taking the upper triangular factor of the QR decomposition of the last factor on the right-hand side.
For a complete algorithm, it remains to obtain the initial approximations of $\Phi_0(A)$ and $U_0(A, B)$.
An approach for this is developed in the following.

\subsection{The initial approximations}\label{subsec:initialization}
In view of the doubling formula in \eqref{eq:exp-and-gram-doubling},
it remains to obtain initial approximations $\widehat{\Phi}_0$ and $\widehat{G}_0$ of the matrix exponential $\Phi_0$ and
the Gramian $G_0$, respectively.
For this purpose, define $A_s = A / 2^s$ and $B_s = B / \sqrt{2^s}$.
Then approximations are sought for
\begin{align*}
\Phi_0(A) &= \Phi(A_s), \\
G_0(A,B)  &= G( A_s, B_s).
\end{align*}
Consider an order $q$ expansion of $t \mapsto \Phi(A_s t)$ in terms of Legendre polynomials $P_k$ on the unit interval:
\begin{equation}\label{eq:legendre-expansion}
\widehat{E}_q(t) = \sum_{k=0}^q C_k(A_s) P_k(t) \approx \Phi(A_s t) .
\end{equation}
Note that, in particular, $\Phi(A_s)$ is approximated by $\widehat{E}_q(1) = \sum_{k=0}^q C_k(A_s)$, since $P_k(1) = 1$. There are various ways of constructing the coefficients $C_k(A_s)$ in \eqref{eq:legendre-expansion},
such as projections in $\mathscr{L}_2([0, 1])$ or various Petrov--Galerkin type methods.
In this paper, a specific Petrov--Galerkin method is used for which $\sum_{k=0}^q C_k(A_s)$ coincides with the usual diagonal Pad\'e approximation to $\Phi(A_s)$~\cite{Moore2011}. It requires that the coefficients satisfy $C_k(A_s) = (2k+1)\tilde{C}_k$ where
\begin{equation}\label{eq:Ck_Moore}
\begin{bmatrix}
\mathrm{I} & - 3 \mathrm{I} & 5 \mathrm{I}  & \ldots & \ldots                & \ldots \\
-A_s       & 6 \mathrm{I}   & A_s          & 0      & \ldots                & \vdots \\
0          & -A_s           & 10 \mathrm{I} & \ddots & \ddots                & \vdots \\
\vdots     & \ddots         & \ddots        & \ddots & A_s                & \vdots \\
\vdots     & \ddots         & \ddots        & \ddots & (4q - 2) \mathrm{I}   & 0      \\
0          & \ldots         & \ldots        & 0      & -A_s                  & (4q + 2) \mathrm{I}
\end{bmatrix}
\begin{bmatrix}
\tilde{C}_0 \\ \tilde{C}_1 \\ \vdots \\ \vdots \\ \tilde{C}_{q-1} \\ \tilde{C}_q
\end{bmatrix}
=
\begin{bmatrix}
\mathrm{I} \\ 0 \\ \vdots \\ \vdots \\ 0 \\ 0
\end{bmatrix}.
\end{equation}
If this holds, then
they are rational functions in $A_s$, say,
\begin{equation*}
C_k(A_s) = D_q^{-1}(A_s) L_k(A_s).
\end{equation*}
Their sum evaluates to
\begin{equation} \label{eq:def_rq}
\sum_{k=0}^q C_k(A_s) = D_q^{-1}(A_s) N_q(A_s) = r_q(A_s),
\end{equation}
where $N_q$ and $D_q$ are the numerator and denominator, respectively, in the diagonal Pad\'e approximation $r_q$ of $e^z$,
see~\cite{Moore2011} and references therein\footnote{
Unless it is necessary for the clarity of exposition, the explicit dependence on $A_s$ in $C_k(A_s)$, $D_q(A_s)$ and $N_q(A_s)$ will henceforth be suppressed.
}.
Tables of coefficients for computing $\{C_k\}_{k=0}^{q}$, $D_q$ and $N_q$ have been generated symbolically in Julia
using the Symbolics.jl package \cite{Gowda2022}.
They are provided in Appendix \ref{sec:coefficient_tables} for $q = 3, 5, 7, 9, 13$.

By replacing $\Phi(A_s \tau)$ by $\widehat{E}_q(\tau)$ in~\eqref{eq:gramian} and using that $\{P_k\}_{k=0}^q$ is a sequence of orthogonal functions which satisfy $\norm{P_k}^2 =  \frac{1}{2k+1}$, it then follows that $\Phi_0(A)$ and $G_0(A, B)$ can be approximated by
\begin{subequations}\label{eq:initial-approximation}
\begin{align}
\widehat{\Phi}_0(A) &= \widehat{E}_q(1) =  \sum_{k=0}^q C_k \label{eq:initial-exp}, \\
\widehat{G}_0(A, B)  &=   \sum_{k=0}^q \frac{1}{2k + 1}  C_k B_s B_s^* C_k^*,
\end{align}
\end{subequations}
respectively.
Furthermore, the Gramian approximation may be written as $\widehat{G}_0(A,B) = \tilde{U}_0^* \tilde{U}_0$, where
\begin{equation}\label{eq:initial-cholesky}
\tilde{U}_0^* =
\begin{bmatrix}
C_0 B_s & C_1 B_s / \sqrt{3} & \cdots & C_k B_s / \sqrt{2k + 1} & \cdots & C_q B_s / \sqrt{2q + 1},
\end{bmatrix}
\end{equation}
and an upper triangular (not necessarily square) square-root factor of $\widehat{G}_0(A, B)$ may be obtained by the QR factorization.

\section{Backward error analysis}\label{sec:backward-error-analysis}
In this section, error analysis of the Gramian approximation is performed.
Throughout the paper, the unit round-off for the chosen finite precision arithmetic will be denoted by $u$.
In most applications, this will be $2^{-53} \approx 1.1 \cdot 10^{-16}$, corresponding to IEEE double precision,
but most of the analysis is agnostic to the specific value.

From the analysis of \cite{Moore2011}, it follows that
\begin{align*}
V_t(A_s)             &= - A_s C_q \int_0^t e^{-A_s\tau} P_q(\tau) \dif \tau, \\
\widehat{E}_q(t) &= e^{A_s t} \big( \mathrm{I}  + V_t(A_s) \big),
\end{align*}
where for each $t$, $V_t(A_s)$ is a matrix function in $A_s$.
When $\norm{V_t(A_s)} < 1$ a backwards error of $\widehat{E}_q$ is given by
\begin{equation*}
F_t(A_s) = \log\big( \mathrm{I}  + V_t(A_s) \big),
\end{equation*}
and since $V_t(A_s)$ is a matrix function in $A_s$, then so is $F_t(A_s)$.
Furthermore, $F_1(A_s)$ is the corresponding backward error for the Pad\'{e} approximation to the matrix exponential,
since $\widehat{E}_q(1) = r_q(A_s)$ by \eqref{eq:initial-exp} and \eqref{eq:def_rq}.
Before proceeding, recall the following definitions from Higham~\cite{Higham2005}.

\begin{itemize}
\item $\theta_q$ is a number such that $2^s \norm{F_1(A_s)} \leq \norm{A}u$ whenever $\norm{A_s} \leq \theta_q$.
\item $\nu_q$ is the maximal radius around the origin for which the Pad\'e denominator, $D_q(z)$, is analytic; $\nu_q = \min \{ z\colon D_q(z) = 0\}$.
\item $\norm{D_q^{-1}(A_s)} \leq \xi_q$ whenever $\norm{A_s} \leq \theta_q$.
\end{itemize}
The following result was obtained in~\cite{Higham2005}, where $\theta_q$ was tabulated for $u = 2^{-53}$.
\begin{proposition}\label{prop:higham2005}
Let $s \geq \max\big(0, \log_2 \frac{\norm{A}}{\theta_q}  \big)$, then
\begin{equation*}
e^{-A_s} r_q(A_s) = \mathrm{I} + V_1(A_s) = e^{F_1(A_s)},
\end{equation*}
where $V_1(A_s)$ and $F_1(A_s)$ are well-defined matrix functions of $A_s$ which commute with $A_s$.
Furthermore, $\widehat{\Phi}(A) = e^{A + 2^s F_1(A_s)}$ and
\begin{equation*}
\frac{\norm{2^s F_1(A_s)}}{\norm{A}} \le u \hide{\approx 1.1 \cdot 10^{-16}},
\end{equation*}
so that $\widehat{\Phi}(A) = r_q^{2^s}(A_s)$ approximates $e^{A}$ to full accuracy in finite precision arithmetic, in the backward error sense.
Lastly, $\theta_q \leq \nu_q$ for $q \leq 21$.
\end{proposition}
This result applies to the present approximation to the matrix exponential as it is computed in exactly the same way.
However, the error analysis for the Gramian is more involved and shall be pursued in the following. As in~\cite{Higham2005}, algorithm parameters are  tabulated for $u = 2^{-53}$, but the analysis can easily be adapted to a different choice of finite precision arithmetic.

\subsection{Backwards error of Gramian}
It appears infeasible to obtain a backwards error for $G$ directly.
However, the problem can be \emph{lifted} to the computation of the Gramian functional \eqref{eq:gramian-functional},
whose definition is restated:
\begin{equation*}
  \mathcal{G}(\tilde{A}, \tilde{B}) = \int_0^1 e^{ \tilde{A}(t)} \tilde{B}(t) \tilde{B}^*(t) e^{\tilde{A}^*(t)} \dif t .
\end{equation*}
Then $\mathcal{G}(t \mapsto At, t \mapsto B) = G(A, B)$, so if there exist matrix-valued functions $\widehat{A}(t)$ and $\widehat{B}(t)$ such that
$\widehat{G}(A, B) = \mathcal{G}(\widehat{A}(t), \widehat{B}(t))$ a backward result is obtained.
The  backward errors $\Delta A$ and $\Delta B$ given by $\Delta A(t) = \widehat{A}(t) - At$ and $\Delta B(t) = \widehat{B}(t) - B$ can then be controlled in the induced supremum norm defined as
\begin{equation*}
\fnorm{f} = \sup_{t \in [0,1]} \norm{f(t)},
\end{equation*}
where $\norm{\cdot}$ is a given matrix norm.
A subscript is added for specific matrix norms, e.g.\ $\fnormtwo{f} = \sup_{t \in [0,1]} \normtwo{f(t)}$.
It should be noted that specifically $\fnorm{At} = \norm{A}$, and $\fnorm{B} = \norm{B}$. Here, and in the following, $At$ and $B$ refers to the functions $t \mapsto At$ and $t \mapsto B$, respectively,
rather than their corresponding function values at a specific $t$. This abuse of notation makes many of the equations significantly less cumbersome and should not lead to confusion.

The starting point for obtaining the perturbed data $\widehat{A}$ and  $\widehat{B}$ is to use the doubling formula and the discrete variation of constants formula so that $\widehat{G}$ is given by
\begin{equation}\label{eq:gramian-approximation}
\widehat{G} = \sum_{m=0}^{2^s - 1} r_q^m(A_s) \widehat{G}_0 r_q^m(A_s^*).
\end{equation}
The idea is then to use this formula to construct $\widehat{A}(t)$ and $\widehat{B}(t)$ such that
\begin{equation*}
\widehat{G} = \mathcal{G}(\widehat{A}(t), \widehat{B}(t)).
\end{equation*}
This is done by dividing the interval $[0, 1]$ into $2^s-1$ equally large sub-intervals of length $2^{-s}$ and then matching terms. It holds that
\begin{equation*}
\begin{split}
  &\mathcal{G}(\widehat{A}(t), \widehat{B}(t)) \\
  &\quad= \int_0^1e^{ \widehat{A}(t) } \widehat{B}(t)\widehat{B}^*(t) e^{\widehat{A}^*(t)}\dif t \\
&\quad= \sum_{k=0}^{2^s-1} \int_{k2^{-s}}^{(k+1)2^{-s}} e^{ \widehat{A}(t) } \widehat{B}(t)\widehat{B}^*(t) e^{\widehat{A}^*(t)} \dif t \\
&\quad= \sum_{k=0}^{2^s-1} 2^{-s}\int_0^1 e^{ \widehat{A}(k2^{-s} + t2^{-s}) } \widehat{B}(k2^{-s} + t2^{-s})\widehat{B}^*(k2^{-s} + t2^{-s}) e^{\widehat{A}^*(k2^{-s} + t2^{-s})} \dif t. \\
\end{split}
\end{equation*}
A match can thus be obtained by first setting
\begin{equation*}
\begin{split}
&2^{-s}\int_0^1 e^{ \widehat{A}(k2^{-s} + t2^{-s}) } \widehat{B}(k2^{-s} + t2^{-s})\widehat{B}^*(k2^{-s} + t2^{-s}) e^{\widehat{A}^*(k2^{-s} + t2^{-s})} \dif t\\
&\quad = r_q^k(A_s)  \Big(  \int_0^1 e^{A_s t} (e^{F_t(A_s)} B_s) (e^{F_t(A_s)} B_s)^* e^{A_s^* t} \dif t  \Big) r_q^k(A_s^*)\\
&\quad =2^{-s} e^{A k2^{-s} + kF_1(A_s)}  \Bigl(  \int_0^1 e^{At2^{-s}} (e^{F_t(A_s)} B) (e^{F_t(A_s)} B)^* e^{A^* 2^{-s}t} \dif t  \Bigr) e^{A^* k2^{-s} + kF_1^*(A_s)},
\end{split}
\end{equation*}
for $k = 0,1,\ldots, 2^s -1$, and then setting
\begin{equation*}
e^{ \widehat{A}(k2^{-s} + t2^{-s}) } \widehat{B}(k2^{-s} + t2^{-s}) =  e^{Ak2^{-s} + kF_1(A_s) + At2^{-s}} e^{F_t(A_s)} B,
\end{equation*}
for $t \in [0, 1)$, or, by the change of variables $t \mapsto k2^{-s} + t2^{-s}$,
\begin{equation}\label{eq:matching-condition}
e^{ \widehat{A}(t) } \widehat{B}(t) =  e^{At + kF_1(A_s)} e^{F_{2^s t-k}(A_s)} B, \quad t \in [k2^{-s}, (k + 1)2^{-s}).
\end{equation}
From this, a backwards result for the approximated Gramian follows quickly, as formalized in the following proposition.
\begin{proposition}\label{prop:gramian-backward-error}
Let $\fnorm{V_t(A_s)} < 1$. Then
\begin{equation*}
\widehat{G}(A, B) = \mathcal{G}(At + \Delta A, B + \Delta B),
\end{equation*}
where
\begin{align*}
\Delta A(t) &= kF_{1}(A_s), \\
\Delta B(t) &= ( e^{F_{2^s t-k}(A_s)}  - \mathrm{I}) B,
\end{align*}
for $t \in [k2^{-s}, (k + 1)2^{-s})$ and $k = 0, \ldots, 2^s-1$.
Furthermore, the relative errors $\Delta A(t)$ and $\Delta B(t)$, respectively, are bounded by
\begin{align*}
\frac{\fnorm{\Delta A(t)}  }{\fnorm{At}} &\leq 2^s\frac{\norm{F_1(A_s)}}{\norm{A}}, \\
\frac{\fnorm{\Delta B(t)}  }{\fnorm{B}}  &\leq \fnorm{V_t(A_s)}.
\end{align*}
\end{proposition}

\begin{proof}
That $\widehat{A}(t) = At + \Delta A(t)$ and $\widehat{B}(t) = B + \Delta B(t)$ satisfy the matching condition \eqref{eq:matching-condition} is readily verified.
Furthermore, a bound on the relative error of $\Delta A(t)$ is given by
\begin{equation*}
\begin{split}
\frac{\fnorm{\Delta A(t)}  }{\fnorm{At}}
&= \frac{\max_{k} \sup_{t \in [k 2^{-s}, (k+1)2^{-s}]} \norm{kF_1(A_s)} }{\norm{A}}
= \frac{\max_{k} \norm{kF_1(A_s)} }{\norm{A}} \\
&= (2^s - 1) \frac{\norm{F_1(A_s)}}{\norm{A}} \leq 2^s \frac{\norm{F_1(A_s)}}{\norm{A}}.
\end{split}
\end{equation*}
Finally, a bound for the relative error in $\Delta B$ is found by
\begin{equation*}
\begin{split}
\frac{\fnorm{\Delta B(t)}}{\fnorm{B}}
&= \max_k \sup_{t \in [k 2^{-s}, (k+1)2^{-s}]} \norm{( e^{F_{2^s t-k}(A_s)}  - \mathrm{I}) B } \\
&= \frac{\fnorm{ (e^{F_t(A_s)}  - \mathrm{I} ) B }}{ \norm{B}  }
\leq  \fnorm{e^{F_t(A_s)}  - \mathrm{I}}
= \fnorm{V_t(A_s)}.
\end{split}
\end{equation*}
\end{proof}

\subsection{Controlling the backward error of the approximated Gramian}
In view of Propositions~\ref{prop:higham2005} and~\ref{prop:gramian-backward-error}, it remains to find a maximal value $\eta_q$ such that $\norm{A_s} < \eta_q$ implies
\begin{equation*}
\fnorm{V_t(A_s)} < u.
\end{equation*}
The scaling parameter $s$ may then be selected as
\begin{equation*}
s = \Big\lceil \log_2 \frac{\norm{A}}{\min(\eta_q,\theta_q)} \Big\rceil,
\end{equation*}
to ensure that both the errors in the approximated matrix exponential and the Gramian are at most unit round-off (in the chosen finite precision arithmetic).
Before proceeding, the following result which gives a more explicit expression for $V_t$ is required.
\begin{proposition}\label{prop:explicit-V}
For the last coefficient $C_q$ in the Legendre expansion of $e^{A_s t}$ it holds that
\begin{equation*}
C_q = \frac{q!}{(2q)!} (-A_s)^q D_q^{-1}(A_s)
\end{equation*}
and as a consequence an explicit expression for $V_t$ is given by
\begin{equation*}
V_t(A_s) =  \frac{q!}{(2q)!} (-A_s)^{q+1} D_q^{-1}(A_s) \int_0^t e^{- A_s \tau} P_q(\tau) \dif \tau.
\end{equation*}
\end{proposition}

\begin{proof}
Since $\widehat{E}_q(1) = r_q(A_s) = e^{A_s} + e^{A_s}V_1(A_s)$ it holds that
\begin{align*}
A_s C_q \int_0^1 e^{A_s(1 - \tau)} P_q(\tau) \dif \tau &= e^{A_s}V_1(A_s) =  e^{A_s} - r_q(A_s) \\
&= \frac{(-1)^q}{(2q)!} A_s^{2q + 1} D_q^{-1}(A_s) \int_0^1 e^{A_s \tau} (1 - \tau)^q\tau^q \dif \tau,
\end{align*}
where the last equality is the Pad\'e remainder \cite{Higham2008}.
By Rodrigues' formula the term $P_q(\tau)$ can be replaced by $\frac{1}{q!} \frac{\dif^{\,q}}{\dif \tau^q}  (\tau^2-\tau)^q$.
Repeated integration by parts and reversing the integration interval thus shows that the left-hand side is given by
\begin{equation*}
\begin{split}
A_s C_q \int_0^1 e^{A_s(1 - \tau)}\frac{1}{q!} \frac{\dif^{\,q} }{\dif \tau^q}  (\tau^2 - \tau)^q \dif \tau
&= A_s C_q \frac{A_s^q}{q!} \int_0^1 e^{A_s(1 -\tau)} (\tau^2 - \tau)^q \dif \tau \\
&=  A_s C_q \frac{A_s^q}{q!} \int_0^1 e^{A_s\tau} \tau^q(1 - \tau)^q \dif \tau.
\end{split}
\end{equation*}
Matching terms with the  Pad\'e remainder then gives the result.
\end{proof}
Proposition \ref{prop:explicit-V} ensures that $V_t$ is analytic in $A_s$ in the neighbourhood $\norm{A_s} < \nu_q$.
Consequently, it has an absolutely convergent power series expansions,
which is more conveniently expressed using the following auxiliary function
\begin{equation*}
\psi(t, \eta) = \frac{q!}{(2q)!} (-\eta)^{q+1} D_q^{-1}(\eta) e^{-\eta t},
\end{equation*}
so that
\begin{equation*}
V_t(A_s) = \int_0^t \psi(\tau, A_s) P_q(\tau) \dif \tau.
\end{equation*}
The function $\psi$ is analytic in the same region as $V_t$ is. Therefore,
\begin{equation*}
\psi(t, \eta) = \sum_{n = q+1}^\infty \psi_n(t) \eta^n,
\end{equation*}
and a bound on the norm of $V_t$ is obtained by
\begin{subequations}
\begin{align}
\bar{V}_n(t) &= \abs[2]{\int_0^t \psi_n(\tau) P_q(\tau) \dif \tau}, \\
\bar{V}(t, \eta) &= \sum_{n =  q + 1}^\infty \bar{V}_n(t) \eta^n, \label{eq:v-bound}\\
\beta(\eta)         &= \sup_{t \in [0, 1]}  \bar{V}(t, \eta), \\
\fnorm{V_t(A_s)} &\leq  \beta(\norm{A_s}).
\end{align}
\end{subequations}
It is clear that $t \mapsto \bar{V}(t, \norm{A_s})$ has extrema at the zeros $z_1, z_2, \ldots, z_q$ of the Legendre polynomial $P_q$.
However, it is not clear that these are the only extrema, even though numerical experiments certainly suggest that this is the case.
In any case, define $z_0 = 0$ and $z_{q+1} = 1$ and form grids constructed by uniformly placing $p$ points in the intervals
$(z_0, z_1), \ldots, (z_q, z_{q+1})$, totalling $(q+1)(p - 1)$ unique points.
Let $t_1, \ldots, t_{(q+1)(p-1)}$ be the union of these grids, and approximate $\beta$ by
\begin{equation*}
\beta(\eta) \approx \hat{\beta}(\eta) = \max_{i}  \bar{V}(t_i, \eta).
\end{equation*}
Maximal positive numbers $\eta_q$ for $q = 1, \ldots, 21$ such that $\hat{\beta}(\eta) \leq u$ when $u = 2^{-53}$ are computed, for $p = 2, 2^6, 2^7$, in the Julia programming language \cite{Bezanson2017},
by using arbitrary precision arithmetic,
truncating the sum \eqref{eq:v-bound} at the 150th order term, and
computing the coefficients $\bar{V}_n$ by Taylor mode automatic differentiation \cite{TaylorSeries.jl-2019}.
The results are tabulated in Table \ref{tab:useful-table} along with the quantities $\theta_q$, $\nu_q$, and $\xi_q$ obtained by \cite{Higham2005}.
As the maximal $\eta_q$ were the same for all selected $p$ only the result for $p = 2$ is presented.
Additionally, the series $\eta_q$ and $\theta_q$ for $q =1, \ldots, 21$ are drawn in Figure~\ref{fig:norm-bounds}.
It is evident that ensuring that the Gramian is computed to double precision round-off accuracy implies a more aggressive scaling of $A$ than for the matrix exponential,
particularly for $2\leq q \leq 11$, while for $q = 1$ and $q \geq 12$ only 2 additional downscalings are required.

The result of the above discussion is summarized in the following theorem:
\begin{theorem}\label{thm:backward_error}
For $q = 1, \ldots, 21$, let $s$ be chosen such that $\norm{A_s} = \norm{A} / 2^s \le \eta_{q}$ according to Table~\ref{tab:useful-table}.
Assume that the coefficients $C_k$ are given by~\eqref{eq:Ck_Moore} and that $\widehat{\Phi}_0(A)$ and $\widehat{G}_0(A,B)$ are given by~\eqref{eq:initial-approximation}.
If $\widehat{\Phi}_k(A)$ and $\widehat{G}_k(A, B)$ are computed by applying the recursion formula~\eqref{eq:exp-and-gram-doubling} starting from $\widehat{\Phi}_0(A)$ and $\widehat{G}_0(A,B)$ then
\begin{equation*}
\widehat{G}_s(A, B) = \mathcal{G}\big(At + \Delta A(t),B + \Delta B(t)\big)
\end{equation*}
with specific perturbations $\Delta A$, $\Delta B$ for which
\begin{equation*}
\max\Bigg(\frac{\fnorm{\Delta A(t)}  }{\fnorm{At}},  \frac{\fnorm{\Delta B(t)}  }{\fnorm{B}}\Bigg) \leq u.
\end{equation*}
\end{theorem}

\begin{table}[h]
\caption{
Maximal values $\theta_q$ and $\eta_q$ of $\norm{2^{-s}A}$ such that $2^s \norm{F_1(A_s)} \leq \norm{A}u$ and
$\protect\fnorm{V_t(A_s)} \leq u$, $\nu_q = \min\{ \abs{z} \colon D_q(z) = 0\}$,
and upper bound $\xi_q$ for $\norm[0]{D_q^{-1}(A)}$. Here, $u = 2^{-53}$.
} \label{tab:useful-table}
\centering
\begin{tabular}{c|ccccccc}
$q$        & 1      & 2      & 3      & 4      & 5      & 6      & 7       \\ \hline
$\theta_q$ & 3.7e-8 & 5.3e-4 & 1.5e-2 & 8.5e-2 & 2.5e-1 & 5.4e-1 & 9.5e-1  \\
$\eta_q$   & 1.8e-8 & 2.4e-5 & 6.7e-4 & 5.3e-3 & 2.1e-2 & 6.0e-2 & 1.3e-1  \\
$\nu_q$    & 2.0e0  & 3.5e0  & 4.6e0  & 6.0e0  & 7.3e0  & 8.7e0  & 9.9e0   \\
$\xi_q$    & 1.0e0  & 1.0e0  & 1.0e0  & 1.0e0  & 1.1e0  & 1.3e0  & 1.6e0   \\
\multicolumn{7}{c}{} \\
$q$        & 8      & 9      & 10     & 11     & 12    & 13    & 14    \\ \hline
$\theta_q$ & 1.5e0  & 2.1e0  & 2.8e0  & 3.6e0  & 4.5e0 & 5.4e0 & 6.3e0 \\
$\eta_q$   & 2.4e-1 & 4.1e-1 & 6.2e-1 & 8.9e-1 & 1.2e0 & 1.5e0 & 1.9e0 \\
$\nu_q$    & 1.1e1  & 1.3e1  & 1.4e1  & 1.5e1  & 1.7e1 & 1.8e1 & 1.9e1 \\
$\xi_q$    & 2.1e0  & 3.0e0  & 4.3e0  & 6.6e0  & 1.0e1 & 1.7e1 & 3.0e1 \\
\multicolumn{7}{c}{} \\
$q$        & 15    & 16    & 17    & 18    & 19    & 20    & 21    \\ \hline
$\theta_q$ & 7.3e0 & 8.4e0 & 9.4e0 & 1.1e1 & 1.2e1 & 1.3e1 & 1.4e1 \\
$\eta_q$   & 2.4e0 & 2.9e0 & 3.4e0 & 4.0e0 & 4.6e0 & 5.2e0 & 5.8e0 \\
$\nu_q$    & 2.1e1 & 2.2e1 & 2.3e1 & 2.5e1 & 2.6e1 & 2.7e1 & 2.8e1 \\
$\xi_q$    & 5.3e1 & 9.8e1 & 1.9e2 & 3.8e2 & 8.3e2 & 2.0e3 & 6.2e3 \\
\end{tabular}
\end{table}

\begin{figure}[h]
\centering
\includegraphics[width=\columnwidth]{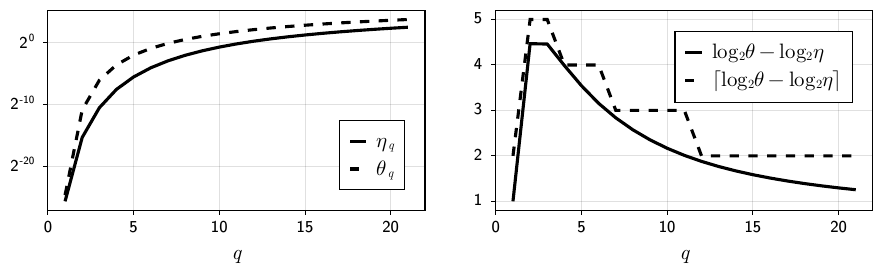}
\caption{The series $\theta_q$ and $\eta_q$ for $q=1, \ldots, 21$ (left), and the difference in the scaling parameter $s$ when selected to
satisfy $\norm{A_s} \leq \theta_q$ or $\norm{A_s} \leq \eta_q$, respectively (right).} \label{fig:norm-bounds}
\end{figure}

\section{Forward error analysis}\label{sec:forward-error-analysis}
According to the backward error analysis, the method exactly solves a problem which is as close to the desired problem as it can be in double precision arithmetic.
Nevertheless, the forward error $\widehat{G}(A, B) - G(A, B)$ can still be large if the problem is ill-conditioned.
The rule of thumb is that the forward error is the backward error multiplied with the condition number of the problem, see e.g.~\cite[Section 1.6]{Higham1996}.
The aim of this section is to obtain estimates on the relative condition number in the spectral norm and the induced supremum norm of $\mathcal{G}$ when evaluated at arguments of the form $(At, B)$.

Before proceeding, it is important to note that the backward error in $(At, B)$ is not unique and that Proposition~\ref{prop:gramian-backward-error} only provides an example of a backward error.
Another example is given as follows.
\begin{lemma}
Let $\fnorm{V_t(A_s)} < 1$, then
\begin{equation*}
\mathcal{G}(At + \Delta A(t), B + \Delta B(t)) = \mathcal{G}(At + \Delta \widetilde{A}(t), B),
\end{equation*}
where the alternative backward error $\Delta \widetilde{A}$ is given by
\begin{equation*}
\Delta \widetilde{A}(t) =  k F_1(A_s) + F_{2^st - k}(A_s) = k F_1(A_s) + \log (\mathrm{I} + V_{2^st - k}(A_s))
\end{equation*}
for $t \in [k2^{-s}, (k+1)(2^{-s})$ and $k = 0, \ldots, 2^s - 1$.
Furthermore, the norm of $\Delta \widetilde{A}$ is bounded by
\begin{equation*}
\fnorm{\Delta \widetilde{A}} \leq 2^s \norm{F_1(A_s)}  - \log(1 - \fnorm{V_t(A_s)}).
\end{equation*}
In particular, when $\norm{2^{-s}A} \leq \eta_q$ the alternative backward error is bounded by
\begin{equation*}
\fnorm{\Delta \widetilde{A}} \leq \norm{A}u   - \log(1 - u).
\end{equation*}
\end{lemma}

\begin{proof}
The first part is immediately verified by
\begin{equation*}
\begin{split}
e^{A t + \Delta A(t)}(B + \Delta B(t)) &= e^{A t + \Delta A(t)}( B + V_{2^st - k}(A_s)B    ) = e^{A t + \Delta A(t) + F_{2^st - k}(A_s)}B \\
&= e^{A t + \Delta \widetilde{A}(t)}B.
\end{split}
\end{equation*}
In order to establish the bound, start with the triangle inequality to obtain
\begin{equation*}
\fnorm{\widetilde{A}} \leq \fnorm{\Delta A(t)} + \fnorm{ \log (\mathrm{I} + V_{2^st - k}(A_s))}.
\end{equation*}
Proposition \ref{prop:gramian-backward-error} ensures that the first term is bounded by $2^s \norm{F_1(A_s)}$.
The second term can be controlled by the following inequality \cite{Higham2005}
\begin{equation*}
\norm{\log (\mathrm{I} + V_{2^st - k}(A_s))} \leq - \log(1 - \norm{V_{2^st - k}(A_s)}).
\end{equation*}
Consequently, when $\fnorm{V_t(A_s)} < 1$ the second term is bounded by
\begin{equation*}
\begin{split}
\fnorm{ \log (\mathrm{I} + V_{2^st - k}(A_s))}  &= \sup_{t \in [0, 1]} \norm{\log (\mathrm{I} + V_t(A_s))}
\leq \sup_{t \in [0, 1]} - \log(1 - \norm{V_t(A_s)}) \\
&\leq - \log(1 - \fnorm{V_t(A_s)}).
\end{split}
\end{equation*}
Finally, the norm bound when $\norm{2^{-s}A} \leq \eta_q$ follows from Theorem \ref{thm:backward_error}.
\end{proof}

The benefit of this alternative choice of backward error representation is two-fold. Firstly, the condition number estimation becomes easier to conduct since only one input argument is perturbed.
Secondly, as will be demonstrated, the relative forward error estimate will be independent of $B$.
Recall that the absolute condition number with respect to the first argument is defined by \cite[Section 3.1]{Higham2008}
\begin{equation*}
\begin{split}
\kappa_{\mathsf{abs}} &= \lim_{\varepsilon \to 0} \sup_{ \fnormtwo{\Delta A} \leq \varepsilon} \varepsilon^{-1} \normtwo{\mathcal{G}(At + \Delta A, B) - \mathcal{G}(At, B)}
= \sup_{ \fnormtwo{\Delta A} \leq 1} \normtwo{L_1(At, B, \Delta A)},
\end{split}
\end{equation*}
where $L_1$ is the Fr\'echet derivative of $\mathcal{G}$ with respect to its first input.
It is given by
\begin{equation*}
L_1(At, B, \Delta A) = \int_0^1 L_0(A\tau, \Delta A(\tau)) B B^* e^{A^*\tau} \dif \tau  + \int_0^1 e^{A\tau} B B^*  L_0^*(A\tau, \Delta A(\tau))  \dif \tau,
\end{equation*}
where $L_0$ is the Fr\'echet derivative of matrix exponential, that is \cite[Section 10.2]{Higham2008}:
\begin{equation*}
L_0(M, \Delta M) = \int_0^1 e^{M(1-s)} \Delta M e^{M s} \dif s.
\end{equation*}
From this the absolute condition number is rather immediate.
\begin{proposition}
The absolute condition number of $\mathcal{G}$ with respect to the first argument is given by
\begin{equation*}
\kappa_{\mathsf{abs}} = 2\normtwo{G(A, B)}
\end{equation*}
\end{proposition}

\begin{proof}
Since $L_1(At, B, \Delta A)$ is Hermitian,
\begin{equation*}
\normtwo{L_1(At, B, \Delta A)} =  \sup_{\normtwo{w} = 1} \abs{ w^* L_1(At, B, \Delta A) w}.
\end{equation*}
Furthermore,
\begin{equation*}
w^* L_1(At, B, \Delta A) w = 2 \int_0^1 w^* L_0(A\tau, \Delta A(\tau)) B B^* e^{A^*\tau} w \dif \tau
\end{equation*}
and thus by Cauchy--Schwarz the worst perturbation satisfies
\begin{equation*}
B^* L_0^*(A\tau, \Delta A(\tau)) w \propto B^* e^{A^* \tau } w,
\end{equation*}
which is obtained when $\Delta A(\tau) = \mathrm{I}$ for $\tau \in [0, 1]$. Hence
\begin{align*}
  \sup_{\normtwo{\Delta  A}  \leq 1} \normtwo{L_1(At, B, \Delta A)} &=  \sup_{\normtwo{w} = 1}  \abs{ w^* L_1(At, B, \mathrm{I}) w} = 2  \sup_{\normtwo{w} = 1}  \abs{ w^* G(A, B) w} \\
  &= 2\normtwo{G(A, B)}.
\end{align*}
\end{proof}

With the previous results in mind, if $2^s\norm{A} \leq \eta_q$ then the absolute and relative forward errors defined by
\begin{equation*}
   e_{\mathsf{abs}} = \mathcal{G}(At + \Delta \widetilde{A}(t), B) - \mathcal{G}(At, B)  \quad \text{and} \quad e_{\mathsf{rel}} = e_{\mathsf{abs}} / G(A,B)
\end{equation*}
are estimated by
\begin{subequations}
  \begin{align}
    e_{\mathsf{abs}} &\lessapprox 2\normtwo{G(A, B)} \big(  \norm{A}u   - \log(1 - u) \big)  \approx 2 u\normtwo{G(A, B)} \big(1 + \norm{A} \big), \label{eq:eabs}\\
    e_{\mathsf{rel}} &\lessapprox 2 \big(  \norm{A}u   - \log(1 - u) \big) \approx 2 u  \big(1 + \norm{A} \big).  \label{eq:erel}
\end{align}
\end{subequations}
The approximate inequalities can be turned into true inequalities by adding a term $o(\fnormtwo{\widetilde{A}})$ to the respective right-hand-side, cf.~\cite[Inequality (3.3)]{Higham2008}.

\section{Controllable subspace of the approximated Gramian}\label{sec:rank}
The discussion has hitherto been centred on controlling the error.
However, another important aspect is the rank properties of the Gramian.
The pair $(A, B)$ is said to be completely controllable if the controllability matrix
\begin{equation*}
\mathcal{C}(A, B) =
\begin{bmatrix}
B & A B & \cdots & A^{n-1} B
\end{bmatrix},
\end{equation*}
is of full rank, which is equivalent to the Gramian being of full rank \cite{Anderson2007}.
In fact, $\mathcal{C}(A, B)$ and $G(A, B)$ have the same range, the so-called controllable subspace.
From the perspective of applications in control and estimation it thus interesting to investigate
whether the method produces an approximation, $\widehat{G}$, that mathematically reproduces the controllability properties of $(A, B)$.
There is also a practical interest in such an investigation.
Namely, even if $G$ is computed up to unit round-off in the forward sense,
this need not be the case for the Cholesky factor due to the squaring $U^* U$.
Consider for instance $U_0 \in \mathbb{R}^{(n-1)\times n}$ and a unit vector $e$, and define
\begin{equation*}
U_1 = \begin{bmatrix} U_0 \\ 0^* \end{bmatrix}, \quad U_2 = \begin{bmatrix} U_0 \\ 2^{-27} e^* \end{bmatrix}.
\end{equation*}
Then $U_2^* U_2 - U_1^* U_1 = 2^{-54} e e^*$, which is a negligible difference in double precision.
However, the difference in the entries of $U_1$ and $U_2$ is not.
Therefore, taking controllability into account is a hedge against unnecessary loss of numerical rank in the approximated Cholesky factor.

It follows from the variation of constants representation \eqref{eq:gramian-approximation} and \eqref{eq:initial-cholesky},
that the approximated Gramian may be written as $\widehat{G} =   \widehat{\mathcal{C}}_{q, s} \widehat{\mathcal{C}}_{q, s}^*$ with
\begin{align*}
\tilde{U}_0^* &= D_q^{-1}(A_s)
\begin{bmatrix}
L_0(A_s) B & \cdots & L_k(A_s) B / \sqrt{2k + 1} & \cdots & L_q(A_s) B / \sqrt{2q + 1}
\end{bmatrix},\\
\widehat{\mathcal{C}}_{q,s} &=
\begin{bmatrix}
r_q(A_s)\tilde{U}_0^* & r_q^2(A_s) \tilde{U}_0^* & \cdots & r_q^{2^s-1}(A_s)\tilde{U}_0^*
\end{bmatrix}.
\end{align*}
The approximated Gramian resembles the outer product of the controllability matrix with itself,
except for the fact that the matrix is formed with a basis different from the monomial one.
More specifically, for $k = 0,1, \ldots, q$ and $m = 0, 1, \ldots, 2^s - 1$,
define the functions
\begin{equation*}
e_{k, m}(z) = \frac{1}{ \sqrt{(2k + 1)2^s} } r_q^m(z2^{-s}) D_q^{-1}(z s^{-s}) L_k(z s^{-s}).
\end{equation*}
Then $\widehat{\mathcal{C}}_{q, s}$ is a block matrix consisting of the following blocks:
\begin{equation*}
e_{k, m}(A)B, \quad k = 0,1, \ldots, q, \quad m = 0, 1, \ldots, 2^s - 1.
\end{equation*}
The polynomials $L_k$ are of degree at most $q$, from which it follows that the functions
$\tilde{e}_{k,m}(z) = D_q^{2^s}(z2^{-s}) e_{k,m}(z)$ are polynomials of at most degree $2^s q$.
If there is an $s$ such that they span the space of polynomials of degree at most $n-1$, then
there is an invertible matrix $T$ such that
\begin{equation}\label{eq:controllable_subspace_requirement}
D_q^{2^s}(A_s) \widehat{\mathcal{C}}_{q, s} T = \begin{bmatrix} \mathcal{C}(A, B) & 0 \end{bmatrix}.
\end{equation}
This implies that the controllable subspace of the approximated Gramian is the same as that of the exact Gramian. To make this statement precise, the following lemma is useful:

\begin{lemma}\label{lem:invariant_subspace}
Let $A$ be a square $n\times n$ matrix, $f$ a function that is analytic on a region containing the spectrum of $A$, and $V$ an $A$-invariant subspace,
then $f(A)V \subset V$. Furthermore, if $f(A)$ is invertible then $f(A)V = V$.
\end{lemma}

\begin{proof}
Let $f_k$ be the Taylor series expansion of $f$ around the origin, truncated at the $k$th term, $\bar{v}$ any vector in $V$, and define
\begin{equation*}
v_k = f_k(A) \bar{v} = \sum_{m=0}^k c_k A^k \bar{v}.
\end{equation*}
Clearly, $v_0 = c_0 \bar{v} \in V$. Assume $v_{k-1} \in V$, then
\begin{equation*}
v_k = v_{k-1} + c_k A^k \bar{v},
\end{equation*}
and $c_k A^k \bar{v} \in V$ due to $A$-invariance of $V$.
Since $V$ is a subspace, it is closed under linear combinations. Thus $v_k \in V$, and by induction this holds for any $k \geq 0$.
Furthermore, since $v_k$ converges and $V$ is closed, $v_\infty = f(A) \bar{v}$ exists and lies in $V$.
Lastly, as established $f(A)V \subset V$, but if $f(A)$ is invertible it defines a bijection between $V$ and $f(A)V$ so in fact $V = f(A) V$.
\end{proof}
Since the range of $\mathcal{C}(A, B)$ is $A$-invariant and $D^{-2^s}$ analytic in a region containing the spectrum of $A$, Lemma~\ref{lem:invariant_subspace} implies that
if there is an invertible $T$ such that \eqref{eq:controllable_subspace_requirement} holds, then
\begin{equation*}
\operatorname{range} \widehat{\mathcal{C}}_{q, s} = \operatorname{range} \widehat{\mathcal{C}}_{q, s} T = \operatorname{range}  D_q^{-2^s}(A_s) \mathcal{C}(A, B) = \operatorname{range} \mathcal{C}(A, B).
\end{equation*}
Consequently, it remains to establish the existence of such a $T$ to ensure that the approximated Gramian mathematically reproduces the controllable subspace of the exact Gramian. As noted above, this is true if the polynomials $\tilde{e}_{k,m}(z)$, $m = 0, \ldots, 2^s-1$, span the space of polynomials of degree $n-1$. This is essentially a condition on $s$, and the following assumption is required in order to find the smallest such $s$.

%
\begin{assumption}\label{ass:legendre_coefficients}
The polynomials $L_0, L_1, \ldots, L_q$ are linearly independent.
\end{assumption}
The table of coefficients in Appendix \ref{sec:coefficient_tables} certainly verifies this assumption for $q = 3, 5, 7, 9, 13$.
Furthermore, the following result on the zeros of $N_q$ and $D_q$ shall prove useful.
\begin{lemma}\label{lem:pade-zeros-poles}
The polynomials $N_q$ and $D_q$ have no zeros in common.
\end{lemma}
\begin{proof}
Ehle~\cite[Theorem 2.1, p. 22]{Ehle1969} states that all the zeros of $N_q$ are in the open left half plane.
Therefore, by the well known relation, $D_q(z) = N_q(-z)$, all zeros of $D_q$ are in the open right half plane,
which gives the desired conclusion.
\end{proof}
It remains to study the span of the union of the following sets:
\begin{equation*}
\Pi^m =
\Big\{\tilde{e}_{0, m}(z), \tilde{e}_{1, m}(z), \ldots, \tilde{e}_{q, m}(z) \Big\}, \quad m = 0, 1, \ldots, 2^s-1.
\end{equation*}
\begin{lemma}\label{lem:pade-spaces}
Let Assumption~\ref{ass:legendre_coefficients} hold. Then $\Pi^m$ are sets of linearly independent functions for
$m = 0, 1, \ldots, 2^s-1$.
\end{lemma}
\begin{proof}
Let $v_k$ be some coefficients for the expansion of the zero function in the set $\Pi^m$,
that is
\begin{equation*}
0 = \sum_{k = 0}^q v_k r_q^m(z) L_k(z) D_q^{2^s-1}(z)  \iff 0 = \sum_{k = 0}^q v_k L_k(z),
\end{equation*}
which by assumption is equivalent to $v_k = 0$ for $k = 0, 1, \ldots, q$.
\end{proof}

\begin{proposition}\label{prop:controllability-properties}
For the sets $\Pi^m, \quad m = 0, 1, \ldots, 2^s-1$, the following holds:
\begin{equation*}
\operatorname{dim} \operatorname{span} \cup_{m=0}^{2^s-1} \Pi^m = q 2^s + 1.
\end{equation*}
\end{proposition}

\begin{proof}
The idea is to show that $\operatorname{span} \Pi^m$ and $\operatorname{span} \Pi^{m + l}$ for $l \geq 1$ can only intersect
for $l = 1$ and that the dimension of this intersection is $1$.
The conclusion is then obtained by use of Lemma~\ref{lem:pade-spaces}.
Expanding the zero function in the set $\Pi^m \cup \Pi^{m+l}$ gives
\begin{equation*}
0 = \sum_{k = 0}^q v_k r_q^m(z) L_k(z) D_q^{2^s-1}(z) + \sum_{k = 0}^q w_k r_q^{m+l}(z) L_k(z) D_q^{2^s-1}(z),
\end{equation*}
which is equivalent to
\begin{equation}\label{eq:zero-function}
0 = D_q^l(z) \sum_{k = 0}^q v_k L_k(z) + N_q^l(z)\sum_{k=0}^q w_k L_k(z).
\end{equation}
$D_q$ and $N_q$ are polynomials of degree $q$, which by Lemma~\ref{lem:pade-zeros-poles}
have no zeros in common.
Therefore \eqref{eq:zero-function} is impossible to satisfy for $l \geq 2$ unless $v_k = w_k = 0$ for $k = 0, 1, \ldots, q$.
For $l = 1$ the only possibility is that, for some arbitrary constant $c$,
\begin{align*}
\sum_{k = 0}^q v_k L_k(z) &= \pm c N_q(z),\\
\sum_{k = 0}^q w_k L_k(z) &= \mp c D_q(z).
\end{align*}
which concludes the proof.
\end{proof}
Proposition~\ref{prop:controllability-properties} combined with the discussion around~\eqref{eq:controllable_subspace_requirement} now immediately implies the following theorem:
\begin{theorem}\label{thm:controllable_subspace}
  The exact Gramian $G(A,B)$ and its approximation $\widehat{G}(A,B)$ have the same controllable subspaces if the number of squarings $s$ satisfies
  \begin{equation}\label{eq:full_rank_s}
s \geq \bigg\lceil \log_2 \frac{n-1}{q} \bigg\rceil.
\end{equation}
\end{theorem}

\section{Design of algorithm}\label{sec:algorithm_design}
The goal of this section is to arrive at a final design of the algorithm,
with the discussion in sections \ref{sec:backward-error-analysis} and \ref{sec:rank} in mind.
Other than achieving a backward error of at most unit round-off and an approximated Gramian that
preserves the controllability properties of $(A, B)$,
other design criteria involve minimizing the number of squarings to avoid the over-scaling phenomena,
while also staying as close as possible to the conventional algorithm for the matrix exponential.
Therefore, only the orders $q = 3, 5, 7, 9, 13$ are considered.

\paragraph{Pre-processing}
For the input matrices $A \in \mathbb{R}^{n \times n}$ and $B \in \mathbb{R}^{n \times m}$,
it is assumed that $m \leq n$.
This is not an unreasonable assumption, as otherwise $B$ would be overparametrized.
More specifically, if $m > n$, then the QR decomposition of $B^*$ is given by
\begin{equation*}
B^* = Q_{B^*} \widetilde{B}^*,
\end{equation*}
where $Q_{B^*} \in \mathbb{R}^{m \times n}$ and $\widetilde{B}^* \in \mathbb{R}^{n \times n}$.
It is evident from the definition \eqref{eq:gramian} that
\begin{equation*}
G(A, B) = G(A, \widetilde{B}).
\end{equation*}
It is therefore reasonable to assume that $m \leq n$, at least after initial pre-processing of the matrix $B$.
Furthermore, it is common in implementations of the matrix exponential to use a so-called balancing transform, i.e.\ to select an invertible matrix $T$ and compute $e^A = Te^{T^{-1}AT}T^{-1}$. A similar idea could also be applied in the Gramian case, since the equality
\begin{equation*}
  G(TAT^{-1}, TB) = T G(A, B) T^*
\end{equation*}
holds. However, the discussion of \cite[p. 496]{AlMohy2009b} recommends that balancing should not be used by default,
and this is therefore not included as a pre-processing step of the algorithm proposed here.
Shifting the eigenvalues through a transformation of the form $A \to A - aI$ is not included either, since this is not compatible with the Gramian structure.

\paragraph{Order adaptation}
In view of the conclusions of the backward error analysis, summarized in Theorem~\ref{thm:backward_error},
modifications of the order adaption in the conventional algorithm \cite{Higham2005} are required.
Namely, as $\eta_q < \theta_q$ for $1 \leq q \leq 21$,  the scaling parameter needs to be selected as
\begin{equation*}
s \geq \lceil \log_2 (\norm{A} / \eta_q) \rceil.
\end{equation*}
Furthermore, in view of Theorem~\ref{thm:controllable_subspace}, the scaling parameter also needs to satisfy the bound~\eqref{eq:full_rank_s}.
Therefore, in order to avoid the over-scaling phenomena, it appears numerically advantageous to simply select the smallest $q$
that satisfies the inequalities
\begin{align*}
\norm{A} &\leq \eta_q, \\
n        &\leq q + 1,
\end{align*}
for $q = 3, 5, 7, 9$.
If no such $q$ is found then the order $13$ method is used and the scaling parameter is selected as
\begin{equation*}
s = \Big \lceil \log_2 \max\biggl( \frac{\norm{A}}{\eta_q},  \frac{n-1}{q} \biggr) \Big\rceil.
\end{equation*}

\paragraph{Implementing the initial values}
From the table of coefficients, the polynomials $L_k$ are even for even $k$ and odd for odd $k$
for $q = 3, 5, 7, 9, 13$.
Consequently, the same evaluation strategy as used by Higham~\cite{Higham2005} may be adopted, and
the initial Cholesky factor of the Gramian may be accumulated from $B$ and $A^2 B$ for $q = 3, 5, 7, 9$.
Similarly, for $q = 13$, the initial Cholesky factor is accumulated from $B, A^2B, A^4B, A^6B$.

  \paragraph{Pseudo-code}
  A pseudo-code summary of the complete algorithm is given in Appendix~\ref{sec:app-algorithms}.

\section{Numerical experiments}\label{sec:experiments}
In this section, the numerical performance of the proposed method is examined on a series of test pairs $(A, B)$.
The backward error analysis of Proposition \ref{prop:gramian-backward-error} suggests that $B$ is of little importance and may thus be chosen arbitrarily.
Suitable collections of test matrices for $A$ shall later be defined in various ways.
Unless otherwise stated, a ground-truth is obtained by computing the Gramian via the matrix fraction decomposition \cite{Axelsson2014} in arbitrary precision \cite{Fousse2007},
and projecting the result on the set of Hermitian matrices.
The matrix exponential for the ground-truth is computed using the software package ExponentialUtilities.jl
\footnote{The default implementation of the matrix exponential in LinearAlgebra.jl precludes the use of arbitrary precision floats.}
\cite{DifferentialEquations.jl-2017} with the generic method using a Pad\'e approximation of order 13.
All relative errors are computed in the $2$-norm.
Additionally, the relative errors are compared to the approximate estimate~\eqref{eq:erel} established in Section~\ref{sec:forward-error-analysis}.
When the algorithm behaves in a backward stable and therefore forward stable manner,
the forward errors should not significantly exceed this bound.

\paragraph{Experiment 0} The matrices $A \in \mathbb{R}^{n \times n}$ and $B \in \mathbb{R}^{n \times 1}$ are defined as
\begin{equation*}
  A_{i,j} =
  \begin{cases}
    1, \text{ if } j + 1 = i, \\
    0, \text{ otherwise}
  \end{cases}
  \quad \text{and} \qquad
  B_i =
  \begin{cases}
    1, \text{ if } i = 1, \\
    0, \text{ otherwise}
  \end{cases}.
\end{equation*}
This experiment serves as a ``unit test'' in the sense that $A$ is nilpotent of index $n$ and the base approximations of order $q$ are exact for nilpotent matrices of order $q+1$.
Furthermore, as $e^{A t} B =
\begin{bmatrix}
  1, & t, & t^2/2, & \cdots, & t^{n-1} / (n-1)!
\end{bmatrix}^*
$,
the Gramian and its Cholesky factor can be computed in closed form by switching to the Legendre basis.

The relative errors are shown in Figure~\ref{fig:experiment:integrator}.
It can be seen that the proposed algorithm performs to an acceptable accuracy, just as predicted by the error analysis.
Furthermore, the error in the approximated Cholesky factor is almost half an order of magnitude smaller than that of the Gramian.
\begin{figure}
\centering
\includegraphics[width=\columnwidth]{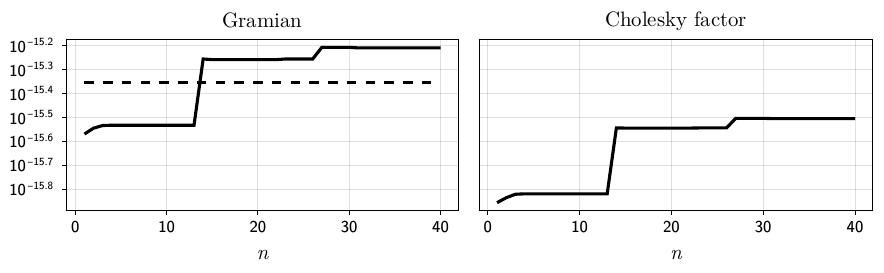}
\caption{
The results of experiment 0. The relative error in the approximated Gramians (left) and Cholesky factors (right).
The dashed line is the approximate error estimate of the computed Gramian given by~\eqref{eq:erel}.
}
\label{fig:experiment:integrator}
\end{figure}
\paragraph{Experiment 1} $A$ is selected from a collection of $10 \times 10$ matrices provided by a subset of the ``builtin'' matrices  of the software package MatrixDepot.jl \cite{Zhang2016}.
The matrices that were excluded were either sparse matrices, not conforming to the general API, or matrices with positive eigenvalues and very large norms,
for which the excessive amount of doublings lead to numerical problems.
The full list of excluded matrices is given in Appendix \ref{sec:app-experiments}.
The matrix $B$ is selected as a $10 \times m$ matrix with $m = 1,5,10$, and the elements drawn independently from the standard Normal distribution, which is then normalized in the induced 2-norm.
The experiment is conducted 50 times for each selection of $m$ so that the effect of randomization can be assessed.
Scatter plots of the relative errors over all simulations are shown in Figure~\ref{fig:experiment:builtin}.
It is again evident that the proposed algorithm performs in accordance with expectation.
One exception is the matrix numbered 22. This matrix is known as, \texttt{"invol"};
it has positive eigenvalues equal to one and a 1-norm resulting in 25 doublings, which is problematic but does not result in complete failure.
The result also demonstrates that the algorithm is rather insensitive to the selection of $B$,
except for a few outliers in the case $m = 1$.

\begin{figure}[t!]
\centering
\includegraphics[width=\columnwidth]{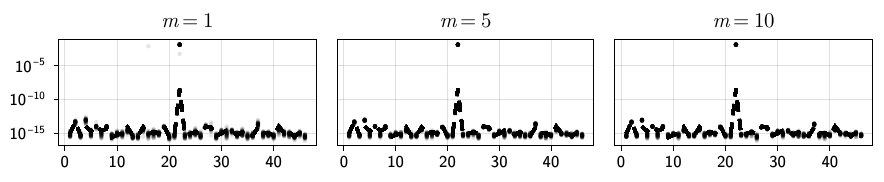}
\caption{
The results of experiment 1. Every value on the horizontal axis corresponds to a specific matrix $A$ from the builtin data set in MatrixDepot.jl.
For each $A$, there are 50 grey dots, each corresponding to the relative error in the approximated Gramian for a single randomly generated matrix $B$.
The dashed line is the approximate error estimate of the computed Gramian given by~\eqref{eq:erel}.
}
\label{fig:experiment:builtin}
\end{figure}

\paragraph{Experiment 2} The matrices $A \in \mathbb{R}^{n\times n}$ and $B \in \mathbb{R}^{n \times 1}$ are defined by
\begin{subequations}\label{eq:laguerre_network}
\begin{align}
A_{i,j} &=
\begin{cases}
- 2\lambda, \quad i > j, \\
-\lambda  , \quad i = j, \\
0, \quad i < j
\end{cases}, \\
B &= \sqrt{2\lambda} \begin{bmatrix} 1 & 1 & \cdots & 1 \end{bmatrix}^*.
\end{align}
\end{subequations}
This is a so called Laguerre network.
The parameter $\lambda$ is a positive number that is selected from $\{1.0, 2.5, 5.0\}$ and $n$ ranges from $1$ to $100$.
$A$ is Hurwitz, and it is readily verified that the identity matrix solves the algebraic Lyapunov equation associated with $(A, B)$, namely
\begin{equation*}
A + A^* = - BB^*.
\end{equation*}
Consequently, the finite horizon Gramian may be computed by the formula \cite[Lemma 1]{Farrell1993}
\begin{equation*}
G(A, B) = I - e^A e^{A^*}.
\end{equation*}
This formula is used to compute a reference solution in arbitrary precision and
the results are shown in Figure~\ref{fig:experiment:laguerre}.
Whereas there is a loss of accuracy as the dimension grows,
the resulting error appears to be acceptable up to dimension at least 100.

\begin{figure}[t!]
\centering
\includegraphics[width=\columnwidth]{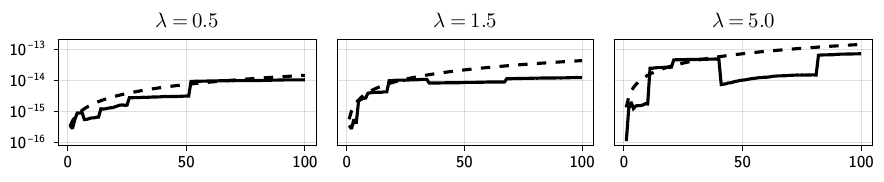}
\caption{
The results of experiment 2.
The relative error in the approximated Gramians, plotted against the dimension of the
Laguerre network \eqref{eq:laguerre_network}.
The dashed line is the approximate error estimate of the computed Gramian given by~\eqref{eq:erel}.
}
\label{fig:experiment:laguerre}
\end{figure}

\section{Conclusions}\label{sec:conclusion}
In this article, a ``scaling and squaring'' method has been developed for computing the Cholesky factor
of the finite horizon Gramian associated with the pair of matrices $(A, B)$.
The method computes the matrix exponential, $e^A$, in almost the same manner as the conventional algorithm \cite{Higham2005}.
Furthermore, a backward error analysis was carried out, which ensures that both the matrix exponential and the Gramian are computed to within unit round-off
in double precision arithmetic.
As the error analysis and algorithm design piggybacks on the development of the conventional algorithm for computing the matrix exponential,
it is expected that the algorithm for the Cholesky factor of the Gramian will have similar numerical performance in practice.
This has indeed been demonstrated through a set of experiments.
The doubling recursion for both the matrix exponential and the Gramian both require one matrix multiplication each,
and the former additionally requires a QR decomposition.
Consequently, the doubling phase is expected to be equally problematic for both quantities.
Nevertheless, in the numerical experiments the algorithm has performed to satisfaction.

\renewcommand{\theHsection}{A\arabic{section}} 
\appendix

\section{Pseudo-code algorithms} \label{sec:app-algorithms}
A pseudo-code summary of the complete algorithm for $q = 3, 5, 7$ and $9$ is given in Algorithms~\ref{alg:init} and~\ref{alg:main}. Algorithm~\ref{alg:init} demonstrates how to compute the initial approximations $\widehat{\Phi}_0(A)$ and $\widehat{U}_0(A,B)$ in an efficient way using the pre-computed coefficients in Appendix~\ref{sec:coefficient_tables}. The algorithm minimizes the number of multiplications of $A$ by using the polynomial evaluation strategy mentioned above. Note that the parts involving $U$ and $V$ are the same as for the matrix exponential~\cite{Higham2005}. For $q=13$, the algorithm can be further optimized by also precomputing $A^4B$ and $A^6B$. This has been done in the Julia package FiniteHorizonGramians.jl\footnote{
\url{https://github.com/filtron/FiniteHorizonGramians.jl}.
}, but is omitted from the pseudo-code for brevity.
Algorithm~\ref{alg:main} is the driver algorithm, which selects the appropriate scaling and performs the doubling. Like the scaling and squaring method for the matrix exponential, it benefits from the availability of an efficient estimate of $\norm{A}$, here denoted $\texttt{normest}(A)$. The current implementation computes the actual norm instead of an estimate, since this was sufficiently fast for the matrices considered in Section~\ref{sec:experiments}.

\captionof{algorithm}{Computing $\widehat{\Phi}_0(A)$ and $\widehat{U}_0(A,B)$. \label{alg:init}}
\begin{algorithmic}
  \Procedure {ExpAndGramInit}{$A$, $B$, $q$}
    \State Note: 1-based indexing
    \State Load coefficients $b_{k}$ from the relevant vector \texttt{pade\_num} in Appendix~\ref{sec:coefficient_tables}
    \State Load coefficients $c_{k,j}$ from the relevant matrix \texttt{leg\_nums} in Appendix~\ref{sec:coefficient_tables}
    \State
    \State $A_2 \gets A^2$
    \State $P \gets A_2$
    \State $V \gets b_3 P + b_1 I$ \Comment Contains even powers of $A$
    \State $U \gets b_4 P + b_2 I$ \Comment Contains odd powers of $A$

    \State $(L^{even}_k, L^{odd}_k) \gets (0_{n \times m}, 0_{n \times m})$ for $k = 1, \ldots, (q+1)/2$ \Comment Initialize with zero matrices
    \State $L^{even}_1 \gets PBc_{1,3} + B c_{1,1}$
    \State $L^{even}_2 \gets PBc_{3,3}$

    \If{$q<4$}
      \State $L^{odd}_1 \gets Bc_{2,2}$
    \Else
      \State $L^{odd}_1 \gets PBc_{2,4} + Bc_{2,2}$
    \EndIf
    \State $L^{odd}_2 \gets  PBc_{4,4}$

    \State

    \For{$k = 2, \ldots, (q-1)/2$}
      \State $P \gets A_2P$
      \State $V \gets b_{2k+1} P V $
      \State $U \gets b_{2k+2} P U $
      \For{$i = 1, \ldots, (q+1)/2$}
        \State $L^{even}_i \gets L^{even}_i + PBc_{2i-1, 2k+1}$
        \State $L^{odd}_i \gets L^{odd}_i + PBc_{2i, 2k+2}$
      \EndFor
    \EndFor
    \State $U \gets A U $

    \State $L^{odd}_k \gets AL^{odd}_k$ for $k = 1, \ldots, (q+1)/2$

    \State
    \State $N_q \gets V + U$
    \State $D_q \gets V - U$

    \State

    \State Solve $D_qL = \begin{bmatrix}  L^{even}_1, & L^{even}_2, & \ldots, & L^{even}_{(q+1)/2}, & L^{odd}_1, & L^{odd}_2, & \ldots, & L^{odd}_{(q+1)/2}  \end{bmatrix}$

    \State $(Q,R) = \texttt{qr}(L)$ \Comment QR-factorize $L$

    \State $\widehat{U}_0(A,B) \gets R$
  \EndProcedure
\end{algorithmic}

\captionof{algorithm}{Computing $\widehat{\Phi}(A)$ and $\widehat{U}(A,B)$. \label{alg:main}}
\begin{algorithmic}
  \Procedure {ExpAndGram}{$A$, $B$}
    \State $\texttt{nrm}  \gets \texttt{normest}(A)$ \Comment Estimate $\norm{A}$
    \For {$q \in \{3, 5, 7, 9\}$}
      \If {$\texttt{nrm} \leq \eta_q$ and $n \le q + 1$ } \Comment No need to rescale
        \State \textbf{return} $\texttt{ExpAndGramInit}(A, B, q)$
      \EndIf
    \EndFor
    \State $q \gets 13$
    \State $s \gets \Big \lceil \log_2 \max\Big( \frac{\texttt{nrm}}{\eta_q},  \frac{n-1}{q} \Big) \Big\rceil$

    \State $A_s \gets A / 2^s$
    \State $B_s \gets B / \sqrt{2^s}$

    \State $(\Phi, U) \gets \texttt{ExpAndGramInit}(A, B, q)$

    \For{$j=1,\ldots, s$} \Comment Doubling phase
      \State $\Phi \gets \Phi^2$
      \State $\tilde{U} \gets \begin{bmatrix} U \Phi & U \end{bmatrix}^T$
      \State $(Q,R) = \texttt{qr}(\tilde{U})$ \Comment QR-factorize $\tilde{U}$
      \State $U \gets R$
    \EndFor

  \EndProcedure
\end{algorithmic}

\section{Additional information on experiments}\label{sec:app-experiments}
As pointed out in the main text, some matrices from MatrixDepot.jl were excluded from experiment 1.
It was the following matrices:
\begin{verbatim}
  [
    "blur",
    "hadamard",
    "phillips",
    "rosser",
    "neumann",
    "parallax",
    "poisson",
    "wathen",
    "invhilb",
    "vand",
    "golub",
    "magic",
    "pascal",
]
\end{verbatim}
The latter five were problematic in the sense of having eigenvalues with positive real part and very large norms.
The former matrices were excluded on the grounds of not conforming with the general API and were usually sparse matrices.

\section{Coefficient tables for the Legendre expansion of the Matrix exponential}\label{sec:coefficient_tables}
In this section, the necessary quantities to implement the initial approximation of the
matrix exponential and the Gramian are listed for $q = 3, 5, 7, 9, 13$.
Recall that the initial approximation of the matrix exponential is given by the diagonal Pad\'e approximant
\begin{equation*}
r_q(z) = \frac{N_q(z)}{D_q(z)} = \frac{\tilde{N}_q(z)}{\tilde{D}_q(z)},
\end{equation*}
where $N_q$ and $D_q$ are the Pad\'e numerator and denominator, respectively.
The numerator $\tilde{N}_q$ and denominator $\tilde{D}_q$ are scaled versions so that all coefficients are integers.
The coefficients of $\tilde{N}_q$ are listed as \texttt{pade\_num}.
Furthermore, the coefficients $C_k(z)$ are given by
\begin{equation*}
C_k(z) = \frac{\tilde{L}_k(z)}{\tilde{D}_q(z)},
\end{equation*}
where $\tilde{L}_k$ are polynomials whose coefficients are listed as the rows of the matrix referred to as \texttt{leg\_nums}. They are rescaled versions of $L_k$, such that $\frac{\tilde{L}_k(z)}{\tilde{D}_q(z)} = \frac{L_k(z)}{D_q(z)}$.
Lastly, the square norms of the Legendre polynomials are listed as \texttt{sqr\_norms}, that is $1, 3, \ldots, 2k+1, \ldots, 2q+1$.

\subsection{Coefficient tables for $q = 3$}
\begin{scriptsize}
\begin{verbatim}
pade_num = [120, 60, 12, 1]
leg_nums = [120 0 2 0; 0 60 0 0; 0 0 10 0; 0 0 0 1]
sqr_norms = [1, 3, 5, 7]
\end{verbatim}
\end{scriptsize}
\subsection{Coefficient tables for $q = 5$}
\begin{scriptsize}
\begin{verbatim}
pade_num = [30240, 15120, 3360, 420, 30, 1]
leg_nums =
[
30240 0 840 0 2 0
0 15120 0 168 0 0
0 0 2520 0 10 0
0 0 0 252 0 0
0 0 0 0 18 0
0 0 0 0 0 1
]
sqr_norms = [1, 3, 5, 7, 9, 11]
\end{verbatim}
\end{scriptsize}
\subsection{Coefficient tables for $q = 7$}
\begin{scriptsize}
\begin{verbatim}
pade_num = [17297280, 8648640, 1995840, 277200, 25200, 1512, 56, 1]
leg_nums =
[
17297280 0 554400 0 3024 0 2 0
0 8648640 0 133056 0 324 0 0
0 0 1441440 0 11880 0 10 0
0 0 0 144144 0 616 0 0
0 0 0 0 10296 0 18 0
0 0 0 0 0 572 0 0
0 0 0 0 0 0 26 0
0 0 0 0 0 0 0 1
]
sqr_norms = [1, 3, 5, 7, 9, 11, 13, 15]
\end{verbatim}
\end{scriptsize}
\subsection{Coefficient tables for $q = 9$}
\begin{scriptsize}
\begin{verbatim}
pade_num =
[
17643225600,
8821612800,
2075673600,
302702400,
30270240,
2162160,
110880,
3960,
90,
1,
]
leg_nums =
[
17643225600 0 605404800 0 4324320 0 7920 0 2 0
0 8821612800 0 155675520 0 617760 0 528 0 0
0 0 1470268800 0 15444000 0 34320 0 10 0
0 0 0 147026880 0 960960 0 1092 0 0
0 0 0 0 10501920 0 42120 0 18 0
0 0 0 0 0 583440 0 1320 0 0
0 0 0 0 0 0 26520 0 26 0
0 0 0 0 0 0 0 1020 0 0
0 0 0 0 0 0 0 0 34 0
0 0 0 0 0 0 0 0 0 1
]
sqr_norms = [1, 3, 5, 7, 9, 11, 13, 15, 17, 19]
\end{verbatim}
\end{scriptsize}

\subsection{Coefficient tables for $q = 13$}
\begin{scriptsize}
\begin{verbatim}
pade_num =
[
64764752532480000,
32382376266240000,
7771770303897600,
1187353796428800,
129060195264000,
10559470521600,
670442572800,
33522128640,
1323241920,
40840800,
960960,
16380,
182,
1,
]
leg_nums =
[
64764752532480000 0 2374707592857600 0 21118941043200 0 67044257280 0 81681600 0 32760 0 2 0
0 32382376266240000 0 647647525324800 0 3620389893120 0 7449361920 0 5569200 0 1080 0 0
0 0 5397062711040000 0 69390806284800 0 260727667200 0 352716000 0 153000 0 10 0
0 0 0 539706271104000 0 4797389076480 0 12443820480 0 10852800 0 2380 0 0
0 0 0 0 38550447936000 0 245321032320 0 439538400 0 232560 0 18 0
0 0 0 0 0 2141691552000 0 9884730240 0 11938080 0 3344 0 0
0 0 0 0 0 0 97349616000 0 324498720 0 248976 0 26 0
0 0 0 0 0 0 0 3744216000 0 8809920 0 3780 0 0
0 0 0 0 0 0 0 0 124807200 0 197064 0 34 0
0 0 0 0 0 0 0 0 0 3670800 0 3496 0 0
0 0 0 0 0 0 0 0 0 0 96600 0 42 0
0 0 0 0 0 0 0 0 0 0 0 2300 0 0
0 0 0 0 0 0 0 0 0 0 0 0 50 0
0 0 0 0 0 0 0 0 0 0 0 0 0 1
]
sqr_norms = [1, 3, 5, 7, 9, 11, 13, 15, 17, 19, 21, 23, 25, 27]
\end{verbatim}
\end{scriptsize}

\bibliographystyle{siamplain}
\bibliography{refs}

\end{document}